\documentclass[11pt, leqno]{article}

\usepackage{amsmath,amssymb,amsthm, epsfig}  
\usepackage{hyperref}
\usepackage{amsmath}
\usepackage{amssymb}
\usepackage{color}
\usepackage{ulem}
\usepackage[mathscr]{eucal}
\usepackage{mathptmx}
\usepackage[usenames,dvipsnames]{pstricks} 
\usepackage{epsfig}
\usepackage{pst-grad} 
\usepackage{pst-plot} 

\title{Infinity Laplacian equation with strong absorptions}
\author{ Dami\~{a}o J. Ara\'ujo \quad Raimundo Leit\~ao \quad Eduardo V. Teixeira }
\date{}

\def \dist {\mathrm{dist}}

\def \Leb {\mathscr{L}^n}

\newtheorem{theorem}{Theorem}[section] 
\newtheorem{lemma}[theorem]{Lemma}

\newtheorem{corollary}[theorem]{Corollary}

\theoremstyle{definition}

\theoremstyle{remark}

\newtheorem{remark}[theorem]{Remark}

\numberwithin{equation}{section}

\newcommand{\intav}[1]{\mathchoice {\mathop{\vrule width 6pt height 3 pt depth  -2.5pt
\kern -8pt \intop}\nolimits_{\kern -6pt#1}} {\mathop{\vrule width
5pt height 3  pt depth -2.6pt \kern -6pt \intop}\nolimits_{#1}}
{\mathop{\vrule width 5pt height 3 pt depth -2.6pt \kern -6pt
\intop}\nolimits_{#1}} {\mathop{\vrule width 5pt height 3 pt depth
-2.6pt \kern -6pt \intop}\nolimits_{#1}}}

\begin{document}
\maketitle

\begin{abstract}
We study regularity properties of solutions to reaction-diffusion equations ruled by the infinity laplacian operator. We focus our analysis in models presenting plateaus, i.e. regions where a non-negative solution vanishes identically. We obtain sharp geometric regularity estimates for solutions along the boundary of plateaus sets. In particular we show that the $(n-\epsilon)$-Hausdorff measure of the plateaus boundary is finite, for a universal number $\epsilon>0$. 

\medskip

\noindent \textbf{Keywords:} Reaction-diffusion equations, infinity laplacian, regularity

\medskip

\noindent \textbf{AMS Subject Classifications:} 35J60, 35B65
\end{abstract}


\section{Introduction} 

The mathematical analysis of problems involving the infinity Laplacian operator, 
\begin{equation}\label{infty_lap eq}
	\Delta_\infty u := \sum\limits_{i,j} \partial_i u \partial_{ij} u \partial_{j} u = (D u)^TD^2u\; Du, 
\end{equation}
constitutes a beautiful chapter of the modern theory of partial differential equations,  yet far from its denouement. The systematic study of problems involving the infinity laplacian operator has been originated by the pioneering works of G. Aronsson \cite{A1, A2}.  The initial purpose of this line of research is to answer the following natural question: given a bounded domain $O \subset \mathbb{R}^n$ and a Lipschitz function $g \colon \partial O \to \mathbb{R}$, find its best Lipschitz extension, $f$, in the sense that it agrees with $g$ on the boundary and for any $O' \Subset O$, if $f = h$ on $\partial O'$, then $\|f\|_{\text{Lip}(O')} \le \|h\|_{\text{Lip}(O')}$. Such a function $f$ is said to be an absolutely  minimizing Lipschitz extension of $g$ in $O$.  Jensen in \cite{J} has proven that a function in an absolutely  minimizing Lipschitz extension if, and only if, it is a viscosity solution to the homogeneous equation $\Delta_\infty u = 0.$ That is,  the infinity Laplacian rules the Euler-Lagrange equation associated to this $L^\infty$ minimization problem. 

\par
\medskip

Through the years, several different applications of the infinity Laplacian theory emerged in the literature, \cite{CMS, PSS, BCF}, just to cite few. We refer to \cite{ACJ} for an elegant discussion on the theory of absolutely  minimizing Lipschitz extensions.

\par
\medskip

While, existence  and uniqueness of viscosity solution to the homogeneous Dirichlet problem $\Delta_\infty h = 0$, in $O$, $u = g$, on $\partial O$ is nowadays fairly well established,  obtaining improved regularity estimates for infinity harmonic functions remains a major open issue in the theory of nonlinear partial differential equations. The example of Aronsson
$$
	h(x,y) = x^{{4}/{3}} - y^{{4}/{3}}
$$
hints out to one of the most famous conjecture in this field: the first derivatives of infinity harmonic functions should be H\"older continuous with optimal exponent $\frac{1}{3}$. The best results known up to date are due to Evans and Savin,  \cite{Evans_Savin}, who proved that infinity harmonic functions in the plane are of class $C^{1,\alpha}$, for some $0<\alpha \ll 1$, see also \cite{Savin}, and to Evans and Smart, \cite{ES}, who obtained everywhere differentiability for infinity harmonic functions in any dimension.

\par
\medskip

The theory of inhomogeneous infinity laplacian equations $\Delta_\infty u = f(X)$ is more recent and subtle. Lu and Wang in \cite{LW} has proven existence and uniqueness of continuous viscosity solutions to the Dirichlet problem
\begin{equation}\label{NH eq}
	\left \{ 
		\begin{array}{rlll}
			\Delta_\infty u &=& f(X) &\text{ in } O \\
			u &=& g &\text{ on } \partial O,
		\end{array}
	\right.
\end{equation}
provided the source function $f$ does not change sign, i.e. either $\inf f > 0$ or else $\sup f < 0$.   Uniqueness may fail if such a condition is violated, \cite[Appendix A]{LW}. While Lipschitz estimates and everywhere differentiability also hold for a function whose infinity laplacian is bounded in the viscosity sense, see \cite{E}, no further regularity is so far known for inhomogeneous equations. 

\par
\medskip

This current work is devoted to the study of reaction-diffusion models ruled by the infinity Laplacian operator. Namely, for $\lambda > 0$ and $0 \leq \gamma < 3$, let
\begin{equation}\label{Lgamma}
	\mathcal{L}^{\,\gamma}_{\, \infty} \, v := \Delta_{\,\infty} v - \lambda (v^+)^\gamma
\end{equation}
denote the $\infty$-diffusion operator with $\gamma$-strong absorption. The case $\gamma =0$ is related to the infinity-obstacle problem, \cite{RTU}. The constant $\lambda >0$ is called the Thiele modulus, which adjusts the ratio of reaction rate to diffusion--convection rate. Given a bounded domain $\Omega \subset \mathbb{R}^n$, $n\ge 2$, and a continuous, nonnegative boundary value datum $g \in C(\partial \Omega)$, we study existence, uniqueness and regularity issues to the Dirichlet problem  
\begin{equation}\label{eq}
\left\{
\begin{array}{rcc}
\mathcal{L}^{\,\gamma}_{\, \infty} \, u = 0 & \mbox{in} & \Omega \\ 
u = \phi & \mbox{on} & \partial\Omega. \\
\end{array}
\right.
\end{equation}
An important feature in the mathematical formulation of equation \eqref{eq} is the possible existence of plateaus, i.e., a priori unknown regions where the function vanishes identically.  
 
\par
\medskip

Upon establishing existence of a viscosity solution, equation \eqref{eq} can be regarded as a inhomogeneous infinity laplacian equation; however  the corresponding source  function is not bounded away from zero.  Notwithstanding, as a preliminar result, we show uniqueness, up-to-the-boundary continuity, and non-negativeness of viscosity solution to Equation \eqref{eq}, Theorem \ref{existence theorem}. The proof is based on comparison principle methods, proven to hold for the operator $\mathcal{L}_\infty^\gamma$. 

\par
\medskip

The heart of the matter, though, lies on geometric regularity estimates for the solution to Equation \eqref{eq}. While it follows by classical considerations that bounded viscosity solutions are locally Lipschitz continuous, no further smoothness property can be inferred by the existing theory. The main result we show in this work assures that a viscosity solution to Equation \eqref{eq} is {\it pointwisely} of class ${C}^{\frac{4}{3-\gamma}}$ along the boundary of the non-coincidence set, $\partial \{u > 0 \}$, Theorem \ref{thmreg}. 

\par
\medskip

One should notice that for each $0 < \gamma < 3$, the regularity estimate established in Theorem \ref{thmreg} is superior than the optimal $C^{1,\frac{1}{3}}$-estimate, yet to be confirmed (or not), for infinity harmonic functions. Hence, it is clear that such a geometric, improved estimate cannot be extended inwards the non-coincidence set $\{u>0\}$. Nonetheless, such an estimate does enforce rather specific geometric information on the behavior of $u$ near the boundary of the coincidence set. By means of barriers, we show that such an estimate is optimal, Theorem \ref{nondegthm}, in the sense that $u$ detaches from its coincidence set precisely as $\text{dist}^{\frac{4}{3-\gamma}}$. This fact allows us to derive Hausdorff measure estimates  for $\partial \{u>0\}$, Corollary \ref{H est}.  

\par
\medskip

We conclude this introduction by pointing out that similar results can be derived to problems with more general absorption terms: $\Delta_\infty u = f(u)$. We have chosen to present this current article for $f(u) = \lambda (u^{+})^\gamma$ as to highlight the main novelties introduced in our analysis.

\section{Notations}


In this article we shall use classical notations and terminologies, which, for the sake of the readers,  we list below. 

The dimension of Euclidean space in which the equations and problems treated in this article are modeled into will be denoted by $n$. 

Given $\mathscr{O}$ a subset of the $\mathbb{R}^n$, we denote by $\partial\mathscr{O}$ its boundary. For $B_r(X) \subset \mathbb{R}^n$ we denote the open ball of radius $r>0$ centered at $X\in \mathbb{R}^n$. For the vectors $\vec{p}=(p_1,\cdots,p_n)$ and $\vec{q}=(q_1,\cdots,q_n)$, we consider $\langle \vec{p},\vec{q} \rangle$ the standard scalar product in $\mathbb{R}^n$ and $|\vec{p}|:=\sqrt{\langle \vec{p},\vec{p} \rangle}$ its Euclidean norm. The tensor product $\vec{p}\otimes \vec{q}$ denotes the matrix $(p_i \cdot q_j)_{1 \leq i,j \leq n}$.

For a real function $\omega$ defined in a open subset of the $\mathbb{R}^n$, we denote by 
$$
D\omega(X):=(\partial_j \omega(X))_{1 \leq j \leq n} \quad \mbox{and} \quad  D^2\omega(X):=(\partial_{ij} \omega(X))_{1 \leq i,j \leq n}
$$
its gradient and its hessian at the point $X\in \mathbb{R}^n$, where $\partial_i \omega$ is a $i$-th directional derivative of $\omega$ and $\partial_{ij}\omega$ the  $j$-th directional derivative of $\partial_i\omega$.

Fixed a domain $\Omega \subset \mathbb{R}^n$, we will call universal any positive constant that depends only on dimension, $\gamma$ and $\Omega$.

For an operator $G \colon \mathscr{O} \times \mathbb{R}^n \times \mbox{Sym}(n) \to \mathbb{R}$ and a domain $\mathscr{O}\subset \mathbb{R}^n$, a continuous function $\omega \colon \mathscr{O} \to \mathbb{R}$ is called a \textit{viscosity subsolution} of the equation 
\begin{equation}\label{eq not}
	G(X,\omega,D\omega,D^2\omega)=0 \text{ in } \mathscr{O},
\end{equation}
if whenever   $ \varphi \in C^2$ is such that $\omega-\varphi$ has a local maximum at some point $Y \in \mathscr{O}$, then there holds 
$$
	G (Y,\omega(Y),D\varphi(Y),D^2\varphi(Y))\geq 0.
$$ 
Similarly, a continuous function $\omega \colon \mathscr{O} \to \mathbb{R}$ is called a \textit{viscosity supersolution} of equation \eqref{eq not}, if    $ \varphi \in C^2$  is such that $\varphi-\omega$ has a local maximum at some point $Y \in \mathscr{O}$, then there holds 
$$
	G (Y,\omega(Y),D\varphi(Y),D^2\varphi(Y))\leq 0.
$$ 
We say $\omega$ a \textit{viscosity solution} of the $G(X,\omega,D\omega,D^2\omega)=0$ when $\omega$ is both a subsolution and a supersolution. 
\section{Preliminaries}

In this Section we make a preliminar analysis on equation \eqref{eq}. Initially, we point that, for the purposes of this article, the Thiele modulus plays no important role, and hence, hereafter, we shall take $\lambda = 1$.

We start off by verifying that any existing viscosity supersolution to \eqref{eq}, $\mathcal{L}_\infty^\gamma u \le 0$, is nonnegative.  Indeed suppose the open set $\mathscr{O}(u):=\{u<0\}$ were nonempty. Then $u$ would satisfy in $\mathscr{O}(u)$ 
\begin{eqnarray}
\label{Nonneg of u 1}
	\left \{ 
		\begin{array}{rcr}
			\Delta_{\,\infty} u   \leq  0, & \text{ in } & \mathscr{O}(u)\; \\
			u= 0, & \text{ on } & \partial \mathscr{O}(u) .
		\end{array}
	\right.
\end{eqnarray} 
By the classical comparison principle for infinity-harmonic functions, see for instance \cite{J},  $u \geq 0$ in $\mathscr{O}(u)$, which drives us to a contradiction. 

We now briefly comment on existence of a viscosity solution to the Dirichet problem \eqref{eq}. As usual it follows by an application of Perron's method once comparison principle is established. 

Indeed, let us consider the functions $\overline{u}$ and $\underline{u}$, solutions to the following boundary value problems:
\begin{equation}\nonumber
\begin{array}{ccc}
\left\{
\begin{array}{rcc}
\Delta_\infty\, \overline{u}= 0 & \mbox{in} & \Omega, \\
\overline{u} = \phi &  \mbox{on} & \partial\Omega.\\
\end{array}
\right.
&
\mbox{and}
&
\left\{
\begin{array}{rllcc}
\Delta_\infty\, \underline{u} &=& \|\phi\|_{L^\infty(\partial\Omega)}^\gamma & \mbox{in} & \Omega, \\
\underline{u} &=& \phi &  \mbox{on} & \partial\Omega.\\
\end{array}
\right.
\\
\end{array}
\end{equation}
Existence of such solutions follows of standard arguments. We note that $\overline{u}$ and $\underline{u}$ are respectively, supersolution and subsolution to \eqref{eq}. Therefore by Comparison principle, Lemma \ref{comparison principle} below, it is possible, under a direct application of Perron's method, to obtain the existence of a viscosity solution in $C(\overline{\Omega})$ of \eqref{eq}, given by
$$
	u(X):=\inf\{\omega(X)\; \vert\; \omega \;\mbox{is a supersolution of}\; \eqref{eq} \; \mbox{and} \; \underline{u}\, \leq \omega \leq \overline{u}  \; \mbox{in} \;\overline{\Omega}\}.
$$
Uniqueness also follows readily from comparison principle. We state these observations as a Theorem for future references.

\begin{theorem}[Existence and Uniquiness]
\label{existence theorem}
Let  $\Omega\subset \mathbb{R}^{n}$ be a bounded domain and $\varphi \in C\left( \partial \Omega \right)$ be a given nonnegative function. Then there exists a nonnegative function $u \in C\left( \overline{\Omega}\right)$ satisfying \eqref{eq} in the viscosity sense.  Moreover, such a solution is unique.
\end{theorem}

We now deliver a proof for comparison principle for the operator  $\mathcal{L}_\infty^\gamma$. The reasoning is somewhat standard in the  theory of viscosity solutions; we carry out the details for the reader's convenience.

\begin{lemma} 
\label{comparison principle}
Let $u_1$ and $u_2$ be continuous functions in $\overline{\Omega}$ satisfying 
$$
\mathcal{L}^\gamma_\infty \, u_1 \leq 0 \quad \mbox{and} \quad\mathcal{L}^\gamma_\infty  \, u_2 \geq 0 \text{ in } \Omega.
$$ 
If $u_1 \geq u_2$ on $\partial \Omega$, then $u_1 \geq u_2$ inside $\Omega$.
\end{lemma}
 \begin{proof}
 Let us suppose, for the purpose of contradiction, that there exists $M_0>0$ such that $M_0= \sup\limits_{\overline{\Omega}}\left(u_2-u_1 \right)$. For each $\varepsilon>0$ small,  define
$$
M_\varepsilon:= \sup_{\overline{\Omega}\times \overline{\Omega}} \left(u_2(X)-u_1(Y)-\frac{1}{2\varepsilon}|X-Y|^2 \right) < \infty.
$$
Let $(X_\varepsilon,Y_\varepsilon) \in \overline{\Omega}\times \overline{\Omega}$  be a point where the maximum is attained. It follows as in \cite[lemma 3.1]{UG} that
\begin{equation}\label{e0}
\lim\limits_{\varepsilon \to 0} \frac{1}{\varepsilon}|X_\varepsilon-Y_\varepsilon|^2=0, \quad \mbox{and} \quad
\lim\limits_{\varepsilon \to 0} M_\varepsilon = M_0.
\end{equation}
In particular we must have
\begin{equation}\label{e3}
\lim\limits_{\varepsilon \to 0} X_\varepsilon = \lim\limits_{\varepsilon \to 0} Y_\varepsilon =: Z_0
\end{equation}
where $u_2(Z_0)-u_1(Z_0)=M_0$. Moreover, one observes that  
$$
	M_0>0\geq \sup_{\partial\Omega} (u_2-u_1),
$$ 
hence $X_\varepsilon \in \Omega'$ for some interior domain $\Omega' \Subset \Omega$ and $\varepsilon>0$ sufficiently small. Therefore, by \cite[Theorem 3.2]{UG} there exist $\mathscr{M},\mathscr{N} \in \mathscr{S}_n$ with
\begin{equation}\label{e1}
\left(\frac{X_\varepsilon-Y_\varepsilon}{\varepsilon}, \mathscr{M} \right) \in \overline{J}^{2,+}_{\Omega} u_2(X_\varepsilon) \quad \mbox{and} \quad \left(\frac{Y_\varepsilon-X_\varepsilon}{\varepsilon}, \mathscr{N} \right) \in \overline{J}^{2,-}_{\Omega} u_1(Y_\varepsilon)
\end{equation}
such that,
\begin{equation}\label{e2}
-\frac{3}{\varepsilon}
\left(
\begin{array}{cc}
I & 0 \\
0 & I \\
\end{array}
\right)
\leq
\left(
\begin{array}{cc}
\mathscr{M} & 0 \\
0 & \mathscr{N} \\
\end{array}
\right)
\leq
\frac{3}{\varepsilon}
\left(
\begin{array}{cc}
I & -I \\
-I & I \\
\end{array}
\right).
\end{equation}
In particular, $\mathscr{M} \leq \mathscr{N}$. By \eqref{e1} and \eqref{e2}, we obtain
$$
	\begin{array}{lll}
		(u_2(X_\varepsilon)^{+})^\gamma &\leq&  \displaystyle \mathscr{M} \left(\frac{X_\varepsilon-Y_\varepsilon}{\varepsilon} \right) \cdot \left(\frac{X_\varepsilon-Y_\varepsilon}{\varepsilon} \right) \\
		&\leq& \displaystyle \mathscr{N} \left(\frac{Y_\varepsilon-X_\varepsilon}{\varepsilon} \right) \cdot \left(\frac{Y_\varepsilon-X_\varepsilon}{\varepsilon} \right) \\
		&\leq& \displaystyle (u_1(Y_\varepsilon)^{+})^\gamma.
	\end{array}
$$
Therefore,
$$
\left(M_\varepsilon+u_1(Y_\varepsilon)+(2\varepsilon)^{-1} |X_\varepsilon-Y_\varepsilon|^2 \right)^{+} \leq u_1(Y_\varepsilon)^{+}.
$$
By \eqref{e0} and \eqref{e3} and letting $\varepsilon \to 0$ in the estimate above gives
$$
(M_0+u_1(Z_0))^+ \leq u_1(Z_0)^+
$$
which drives us to a contradiction since $u_1 \geq 0$ and $M_0>0$, by assumption.
\end{proof}

\section{Geometric regularity estimates}

As previously mentioned, viscosity solutions to 
\begin{equation}\label{localeq}
\mathcal{L}^{\,\gamma}_\infty \, u  = 0 \quad \mbox{in} \quad \Omega,
\end{equation}
for $0 \leq \gamma < 3$, are locally Lipschitz continuous. This is the optimal regularity estimated available in the literature -- there is hope to show $C^{1,\alpha}$ estimates for some $0<\alpha\le 1/3$, but certainly not beyond that. Surprisingly, in this Section we show a sharp, improved regularity estimate for $u$ along its free plateaus boundary $\partial \{u>0\} \cap \Omega$. The proof is based on a flatness improvement argument inspired by \cite{T, T1}; see also \cite{T2} for improved estimates that hold solely along {\it nonphysical} free interfaces. 

Next Lemma provides a universal way to flatten a solution near a plateaus boundary point. In the sequel we shall apply such a Lemma in dyadic balls as to obtain the aimed regularity estimate at free plateaus boundary points. 

\begin{lemma}[Flattening solutions]\label{comp} Given $\mu>0$, there exists a number $\kappa_\mu>0$, depending only on $\mu$ and dimension such that  if $v \in C(B_1)$ satisfies 
$$
	v(0)=0, \quad 0\leq v \leq 1 \quad \mbox{in} \quad B_1  
$$
and 
$$
	\Delta_{\,\infty} v - \kappa^4 (v^+)^\gamma = 0 \quad \mbox{in} \quad B_1,
$$
for  $0 < \kappa \leq \kappa_\mu$, then
$$
\sup\limits_{B_{1/2}} v \leq \mu.
$$
\end{lemma}

\begin{proof}
Let us suppose, for the sake of contradiction, that there exists $\mu_0>0$ and sequences $\{v_\iota\}_{\iota \in \mathbb{N}}$, $\{\kappa_\iota\}_{\iota \in \mathbb{N}}$ satisfying
$$
0\leq v_\iota \leq 1, \quad v_\iota(0)=0 
$$
and
$$
\Delta_{\,\infty} v_\iota - \kappa_\iota^4 (v_\iota^+)^\gamma=0 \quad \mbox{for} \quad \kappa_\iota = \mbox{o}(1), 
$$
such that,
\begin{equation}\label{supcomp}
\sup\limits_{B_{1/2}} v_\iota > \mu_0.
\end{equation}
By Lipschitz estimates, the sequence $\{v_\iota\}_{\iota \in \mathbb{N}}$ is pre-compact in the $C^{0,1}(B_{1/2})$ topology. Up to a subsequence, $v_\iota \to v_\infty$ locally uniform in $B_{2/3}$. Moreover, we have $v_\infty(0)=0$, $0\leq v_\infty \leq 1$ and
\begin{equation}
\Delta_{\,\infty} v_\infty = 0 \quad \mbox{in} \quad B_1.
\end{equation}
Therefore, by the maximum principle for infinity harmonic functions, we obtain $v_\infty \equiv 0$. This give us a contradiction to \eqref{supcomp}, if we choose $\iota \gg 1$.
\end{proof}

\begin{theorem}\label{thmreg} Let $u$ be a viscosity solution to equation \eqref{localeq} and $X_0 \in \partial\{u>0\} \cap \Omega$.  There exists a positive constant $C>0$ depending on, $\|u\|_{L^{\infty}(\Omega)}$, $(3-\gamma)$ and  $\dist (X_0, \partial \Omega)$, such that 
\begin{equation}\label{conclusion thmreg}
	u(X) \leq C \cdot |X-X_0|^{\frac{4}{3-\gamma}}
\end{equation}
for $X \in \{u>0\}$ near $X_0$.
\end{theorem}

\begin{proof}
We assume, with no loss of generality, that $X_0=0$ and $B_{1} \Subset\Omega$. Let us define
$$
\omega_1(X):= \tau \,u\left( \rho X \right) \quad \mbox{in} \quad B_{1},
$$
for $\tau>0$ and $\rho>0$, constants to be determined universally. From the equation satisfied by $u$, we easily verify that $\omega_1$ satisfies
\begin{equation}\label{eqcomp}
\Delta_\infty\, \omega_1 - \tau^{3-\gamma}\rho^4(\omega_1^+)^{\gamma}=0,
\end{equation}
in the viscosity sense. If $\kappa_\star>0$ is the universal constant granted by previous Lemma \ref{comp} when one takes $\mu=2^{-\frac{4}{3-\gamma}}$, we make the following choices in the definition of $\omega_1$:
$$
\tau := \|u\|_{L^\infty(\Omega)}^{-1} \quad \mbox{and} \quad \rho := \kappa_\mu \cdot \tau^{-\frac{3-\gamma}{4}}.
$$ 
With such a (lucky) selection, $\omega_1$ fits into the framework of Lemma $\ref{comp}$, which ensures that
$$
	\sup_{B_{1/2}} \omega_1 \leq 2^{-\frac{4}{3-\gamma}}.
$$
In the sequel, we set
$$
\omega_{\,2}(X):= 2^{\frac{4}{3-\gamma}} \,\omega_1 \left( 2^{-1}X \right) \quad \mbox{in} \quad B_1.
$$
We note that $\omega_2$ satisfies $\omega_2(0) = 0$, $0 \le \omega_2 \le 1$ and 
$$
\Delta_\infty\, \omega_2 - \kappa_\star^4(\omega_2^+)^{\gamma}=0.
$$
That is, we can apply Lemma \ref{comp} to $\omega_2$ as well, yielding, after rescaling, 
$$
\sup_{B_{1/4}} \omega_{1} \leq 2^{-2 \cdot \frac{4}{3-\gamma}}.
$$
Now, we argue by finite induction. For each $k \in \mathbb{N}$, we define
$$
	\omega_{\,k}(X):= 2^{\frac{4}{3-\gamma}} \,\omega_{\,k-1} \left( 2^{-1}X \right).
$$
By the same reasoning employed above, we  verify that $\omega_{\,k}(X)$ fits into the hypotheses of Lemma \ref{comp}, which gives after rescaling
\begin{equation}\label{sup}
\sup_{B_{2^{-k}}} \omega_1 \leq 2^{-k \,\frac{4}{3-\gamma}}. 
\end{equation}
Finally, fixed a radius $0 < r \leq \dfrac{\rho}{2}$, we choose $k \in \mathbb{N}$ such that, 
$$
2^{-(k+1)} < \frac{r}{\rho} \leq 2^{-k}.
$$
Therefore, we estimate
$$
\sup\limits_{B_{r}} u \leq \sup\limits_{B_{\rho\, 2^{-k}}} u = \tau^{-1}\sup\limits_{B_{2^{-k}}} \omega_1,
$$
yielding, by \eqref{sup},
\begin{equation}
\begin{array}{ccl}
\sup\limits_{B_{r}} u & \leq & \tau^{-1} \cdot 2^{-k \, \frac{4}{3-\gamma}} \\
& \leq & \left(2^{\frac{4}{3-\gamma}}\tau^{-1} \right) \cdot \, 2^{-(k+1) \, \frac{4}{3-\gamma}} \\
& \leq & \left((\rho\tau)^{-1}2^{\frac{4}{3-\gamma}}\right) \cdot r^{\;\frac{4}{3-\gamma}}. \\
\end{array}
\end{equation}
This concludes the proof of Theorem \ref{thmreg}.
\end{proof}

\begin{figure}[h!]
\begin{center}
\scalebox{0.8} 
{
\begin{pspicture}(0,-2.7625)(15.106281,2.7625)
\definecolor{color1}{rgb}{0.9411764705882353,0.9254901960784314,0.9254901960784314}
\definecolor{color663}{rgb}{0.6,0.6,0.6}
\definecolor{color262}{rgb}{0.6509803921568628,0.6235294117647059,0.6235294117647059}
\definecolor{color241}{rgb}{0.5568627450980392,0.5215686274509804,0.5215686274509804}
\definecolor{color471}{rgb}{0.6901960784313725,0.6470588235294118,0.6470588235294118}
\definecolor{color495}{rgb}{0.6627450980392157,0.6078431372549019,0.6078431372549019}
\definecolor{color472}{rgb}{0.6941176470588235,0.611764705882353,0.611764705882353}
\definecolor{color477}{rgb}{0.7294117647058823,0.6549019607843137,0.6549019607843137}
\usefont{T1}{ppl}{m}{n}
\rput(0.6969061,-2.1235113){\color{color1}.}
\usefont{T1}{ppl}{m}{n}
\rput(2.9569058,1.0764886){\color{color1}.}
\psframe[linewidth=0.01,dimen=outer](10.783281,2.7625)(0.003281,-2.7625)
\usefont{T1}{ptm}{m}{n}
\rput(0.8141117,-1.6875){$\{u > 0\}$}
\usefont{T1}{ptm}{m}{n}
\rput(5.5941114,2.0125){$u$}
\psline[linewidth=0.034cm,linecolor=color663](10.765,-1.3575)(0.0,-1.3425)
\usefont{T1}{ptm}{m}{n}
\rput(9.734114,-1.7075){$\{u = 0\}$}
\psline[linewidth=0.04](2.5317678,2.7225)(2.5117679,2.7025)(1.1117679,0.2225)(3.5117679,2.6025)(1.891768,-0.0375)(4.331768,2.6225)(2.7747452,-0.16855443)(4.951768,2.1225)(3.611768,-0.3175)(5.431768,1.6025)(4.311768,-0.5375)(5.951768,0.9625)(5.191768,-0.6975)(6.391768,0.4225)
\usefont{T1}{ptm}{m}{n}
\rput(13.168781,2.0875){\tiny Zoom +}
\psdots[dotsize=0.12](7.741406,-1.3325)
\psframe[linewidth=0.02,linecolor=color241,dimen=outer](13.583,2.2275)(11.123,0.3675)
\psline[linewidth=0.1cm,linecolor=color262](11.123,0.7275)(13.563,0.7275)
\psbezier[linewidth=0.08](11.143,2.2275)(11.443,1.3875)(12.703,0.7275)(13.203,0.7675)
\usefont{T1}{ptm}{m}{n}
\rput(12.685813,1.6325){$\sim d^{\frac{4}{3-\gamma}}$}
\psline[linewidth=0.04](6.371768,0.4225)(5.851768,-0.8175)(6.691768,-0.0175)(6.391768,-0.9775)(7.031768,-0.4375)(6.791768,-1.0575)(7.211768,-0.6975)(7.091768,-1.1175)(7.371768,-0.8575)(7.331768,-1.1575)(7.531768,-1.0175)(7.531768,-1.2375)(7.6517677,-1.1775)(7.711768,-1.3175)
\psframe[linewidth=0.02,linecolor=color241,dimen=outer](7.971768,-1.1725)(7.503,-1.4975)
\psline[linewidth=0.008cm,linecolor=color471](7.531768,-1.1975)(11.111768,2.2025)
\psline[linewidth=0.008cm,linecolor=color472](7.531768,-1.4775)(11.131768,0.3825)
\psline[linewidth=0.008cm,linecolor=color495](7.931768,-1.2175)(11.111768,0.7625)
\psline[linewidth=0.008cm,linecolor=color477](7.931768,-1.4775)(13.571768,0.3825)
\psdots[dotsize=0.14](13.211768,0.7625)
\end{pspicture} 
}
\end{center}
\caption{This picture is a caricature of the improved regularity estimate: by zooming-in around a free boundary point, one sees a $C^{\frac{4}{3-\gamma}}$ surface leading $u$ towards a smooth landing on the plateaus.}
\end{figure}
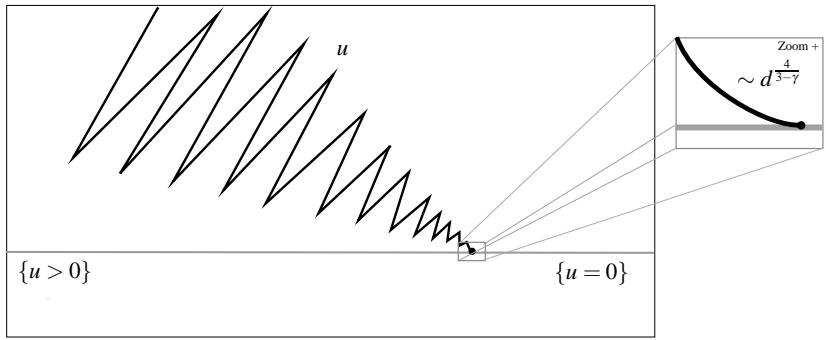

\begin{remark}\label{rmk1} A careful scrutiny of the proof of Theorem \ref{regthm} revels that the same regularity estimate holds for equations with non constant,  bounded Thiele modulus:
$$
	\Delta_\infty u = \lambda(X) \cdot u^\mu.
$$
In this case, the constant $C>0$ appearing in \eqref{conclusion thmreg},  which bounds the $C^{\frac{4}{3-\mu}}$-growth estimate of $u$ away from the touching ground, depends only on $\|u\|_{L^{\infty}(\Omega)}$, $(3-\gamma)$, $\dist (X_0, \partial \Omega)$ and $\|\lambda\|_{L^{\infty}(\Omega)}$. This remark will be used in the future.
\end{remark}

We conclude this Section with an asymptotic Liouville type classification result. A stronger, quantitative version of this Theorem will be delivered later.

\begin{theorem}\label{bernstein}
Let $u$ be a viscosity solution to
$$
\mathcal{L}^{\,\gamma}_{\, \infty} \, u = 0 \quad \mbox{in} \quad \mathbb{R}^n
$$
with $u(0)=0\,$. If $u(X)=\mathrm{o}(|X|^{\frac{4}{3-\gamma}})$ as $|X|\to \infty$, then $u\equiv 0$.
\end{theorem}

\begin{proof}
For each positive number $\kappa \gg 1$, let us define
$$
u_{\kappa}(X):= u(\kappa\, X)\,\kappa^{-\frac{4}{3-\gamma}}.
$$
It is easy to check that
$$
\mathcal{L}^{\,\gamma}_{\, \infty} \, u_\kappa = 0  \mbox{ in } B_1,
$$
and $u_\kappa(0)=0$. Moreover, we note that 
$$
\|u_\kappa\|_{L^\infty(B_1)} = \mbox{o}\,(1).
$$ 
In fact, for each $\kappa \in \mathbb{N}$, let $X_\kappa \in \mathbb{R}^n$ be such that $u_\kappa(X_\kappa)= \sup\limits_{B_1}u_\kappa$. If $\lim\limits_{\kappa \to \infty} \kappa X_\kappa = \infty$, by the above assumption, we obtain
$$
u_\kappa(X_\kappa) \leq |\kappa X_\kappa|^{-\frac{4}{3-\gamma}}u(\kappa X_\kappa) \to 0, \quad \mbox{ as }\kappa \to \infty. 
$$
If the sequence $\{\kappa X_\kappa\}$ remains bounded, we easily obtain the limit above for $u_\kappa(X_\kappa)$. Applying Theorem \ref{thmreg} we obtain
\begin{equation}\label{odeum} 
u_\kappa(X) \leq \mbox{o}\,(1)\cdot |X|^{\frac{4}{3-\gamma}} \quad \mbox{in} \quad B_{1/2}.
\end{equation}
Now, if we assume that there is a $Z_0 \in \mathbb{R}^n$ such that $u(Z_0)>0$, we obtain
from \eqref{odeum}, 
	\begin{equation}\label{Beq05.5}
			\sup\limits_{B_{1/2}} \dfrac{u_\kappa(X)}{|X|^{\frac{4}{3-\gamma}}} \le \dfrac{{u(Z_0)}}{100 |Z_0|^{\frac{4}{3-\gamma}}},
	\end{equation}	
	provided $\kappa \gg 1$. We now estimate, for $\kappa \gg 2|Z_0|$, 
	\begin{equation}\label{Beq06}
			\dfrac{u(Z_0)}{|Z_0|^{\frac{4}{3-\gamma}}} \le \sup\limits_{B_{\kappa/2}} \dfrac{u(X)}{|X|^{\frac{4}{3-\gamma}}} \\
			\le  \sup\limits_{B_{1/2}} \dfrac{u_\kappa(X)}{|X|^{\frac{4}{3-\gamma}}} \\
		 \le  \dfrac{{u(Z_0)}}{100 |Z_0|^{\frac{4}{3-\gamma}}},		 
	\end{equation}	
	which finally drives us to a contradiction, completing the proof of Theorem \ref{bernstein}. 
\end{proof}

 \section{Radial Analysis}

In this intermediary section, we make a short pause as to analyze the radial boundary value problem
\begin{equation}\label{rad eq}
	\left \{ 
		\begin{array}{rlll}
			\Delta_\infty u &=& \lambda (u^{+})^\gamma &\text{ in } B_R(X_0) \\
			u &=& c &\text{ on } \partial B_R(X_0),
		\end{array}
	\right.
\end{equation}
where $0 < c, \lambda < \infty$ are constants and $X_0 \in \mathbb{R}^n$. Herein we  consider an arbitrary Thiele modulus $\lambda>0$, as to amplify the range of our analysis. 

Initially we observe that, by uniqueness and $\mathcal{O}(n)$ invariance of the infinity laplacian, it is plain that the solution of such a boundary value problem is radially symmetric. Indeed, for any $O \in \mathcal{O}(n)$, the function $v(X - X_0) := u(O(X-X_0))$ solves the same boundary value problem, hence, by uniqueness, $v(X) = u(X)$. Since $O \in \mathcal{O}(n)$ was taken arbitrary, it does follow that $u$ is radially symmetric. 

We then consider the following ODE related to \eqref{rad eq}, 
\begin{equation}\label{rad edo}
			h''(h')^2 = \lambda (h^{+})^\gamma \quad \text{ in }\, (0,T)
\end{equation}
satisfying the initial conditions: 	$h(0)=0$ and $h(T)=c$. Solving \eqref{rad edo} we obtain the solution $\,h(s)=\tau(\lambda,\gamma) \cdot s^{\,\frac{4}{3-\gamma}}$, where 
\begin{equation}\label{tau}
	\tau(\lambda, \gamma)=  \sqrt[3-\gamma]{\lambda \cdot \frac{(3-\gamma)^4}{64(1+\gamma)}}
	\quad \mbox{and} \quad 
	\left( \frac{c}{\tau(\lambda,\gamma)} \right)^{\frac{3-\gamma}{4}}=:T. 	
\end{equation}
Fixed $X_0 \in \mathbb{R}^n$ and $0<r<R$, let us assume the {\it dead-core compatibility condition}
\begin{equation}\label{DC cond1}
	R > T.
\end{equation}
Define the following radially symmetric function $u : B_R(X_0) \setminus B_r(X_0) \to \mathbb{R}$ given by
$$
u(X):=h\left(|X-X_0|-r\, \right),
$$
where $r=R-T$. One easily verifies that $u$ solves pointwise the equation
\begin{equation}\nonumber
			\Delta_\infty u =  \lambda (u^{+})^\gamma \quad \text{ in } \quad B_R(X_0) \setminus B_r(X_0).
\end{equation}
The boundary conditions: $u \equiv 0$  on  $\partial B_r$ and $u \equiv c$ on $\partial B_R$ are also satisfied. Moreover, by the construction, for each $Z\in \partial B_r(X_0)$, we obtain
$$
\lim_{X \to Z} \nabla u(X) = h'(0^+).\frac{Z}{|Z|}=0.
$$
Thus, extending $u\equiv 0$ in $B_r(X_0)$, we obtain a function in $B_R(X_0)$ satisfying 
\begin{equation}\nonumber
			\Delta_\infty u =  \lambda (u^{+})^\gamma \quad \text{ in } B_R(X_0).
\end{equation}
We concluded that the function
$$
	u(X):= \tau(\lambda,\gamma) \left(|X-X_0|-R + \left( \frac{c}{\tau(\lambda,\gamma)} \right)^{\frac{3-\gamma}{4}} \right)_+^{\frac{4}{3-\gamma}}
$$
is the solution to \eqref{rad eq}. Its plateaus is precisely $B_r(X_0)$, where
\begin{equation}\label{setup}
0<r:=R-\left( \frac{c}{\tau(\lambda,\gamma)} \right)^{\frac{3-\gamma}{4}}.
\end{equation}

\smallskip 

Let us now deliver few elementary conclusions. Given a positive boundary data $c$, a radius $R>0$, a Thiele modulus $\lambda$,  and an exponent $0\le \gamma < 3$, then
\begin{enumerate}
	\item If the Thiele modulus $\lambda$ is sufficiently large (with bounds easily computable), then the radial boundary problem presents plateaus irrespective of $0\le \gamma <3$.
	\item As one should expect,  solution converges  locally uniform in to zero as $\lambda$ goes to infinity. 
	\item On the other hand, fixed any small Thiele modulus $\lambda_0>0$, the boundary value problem has plateaus provided $\gamma$ is sufficiently close to 3; and indeed, solutions to \eqref{rad eq} go to zero as $\gamma \nearrow 3$. 
\end{enumerate}
   
\begin{figure}[h!]
\begin{center}
\scalebox{0.8} 
{
\begin{pspicture}(0,-4.079871)(11.164836,3.9580135)
\definecolor{color593g}{rgb}{0.8941176470588236,0.8941176470588236,0.8980392156862745}
\definecolor{color593f}{rgb}{0.8666666666666667,0.8862745098039215,0.8862745098039215}
\definecolor{color593e}{rgb}{0.4,0.4,0.4}
\definecolor{color5}{rgb}{0.40784313725490196,0.3058823529411765,0.3058823529411765}
\definecolor{color24}{rgb}{0.9411764705882353,0.9254901960784314,0.9254901960784314}
\definecolor{color423}{rgb}{0.3058823529411765,0.3058823529411765,0.33725490196078434}
\definecolor{color39}{rgb}{0.2980392156862745,0.30196078431372547,0.34509803921568627}
\definecolor{color40}{rgb}{0.40784313725490196,0.28627450980392155,0.28627450980392155}
\pspolygon[linewidth=0.02,shadowcolor=color593e,linecolor=white,shadow=true,fillstyle=gradient,gradlines=2000,gradbegin=color593g,gradend=color593f,gradmidpoint=1.0](1.7024888,0.10496419)(0.0,-3.935036)(9.74,-3.995036)(11.08,0.08496411)
\usefont{T1}{ppl}{m}{n}
\rput(6.573795,-3.0703874){\color{color24}.}
\usefont{T1}{ppl}{m}{n}
\rput(4.3137946,0.12961239){\color{color24}.}
\rput{1.4066979}(-0.045747217,-0.14301053){\psellipse[linewidth=0.01,linecolor=color5,dimen=outer](5.801751,-1.9347274)(3.2207081,0.9394911)}
\psdots[dotsize=0.1,dotangle=-4.763642](5.744052,-1.9043761)
\usefont{T1}{ptm}{m}{n}
\rput(6.289036,-2.464376){\small $R$}
\rput{1.4066979}(-0.04506483,-0.14205708){\psellipse[linewidth=0.04,linecolor=color423,dimen=outer,fillstyle=vlines,hatchwidth=0.028222222,hatchangle=45.0](5.76326,-1.906458)(1.3386139,0.359646)}
\psbezier[linewidth=0.04,linecolor=color39](7.092489,-1.9043763)(8.12,-1.4750359)(8.76,1.6649641)(9.14,3.4649642)
\rput{1.4066979}(0.08498112,-0.1404915){\psellipse[linewidth=0.012,linecolor=color40,linestyle=dashed,dash=0.17638889cm 0.10583334cm,dimen=outer](5.7645197,3.3909202)(3.3611968,0.5217588)}
\psline[linewidth=0.02cm,linestyle=dashed,dash=0.16cm 0.16cm](5.762489,-1.935036)(7.18,-2.735036)
\usefont{T1}{ptm}{m}{n}
\rput(8.958896,-0.77437603){\footnotesize $\sim (|X|-r)_+^{\frac{4}{3-\gamma}}$}
\psline[linewidth=0.02cm,linestyle=dashed,dash=0.16cm 0.16cm](5.82,-1.8750359)(6.86,-1.6950359)
\usefont{T1}{ptm}{m}{n}
\rput(5.900911,-1.709376){\footnotesize \textit{$r$}}
\psline[linewidth=0.0139999995cm,tbarsize=0.07055555cm 5.0]{|*-|*}(1.94,3.344964)(2.1,-2.155036)
\usefont{T1}{ptm}{m}{n}
\rput(1.6690358,0.615624){\small $c$}
\psbezier[linewidth=0.04,linecolor=color39](4.44,-1.9643761)(3.2724888,-1.3150358)(2.8124888,1.5249641)(2.42,3.344964)
\usefont{T1}{ptm}{m}{n}
\rput(1.8434268,-3.594376){\footnotesize Plateaus}
\psline[linewidth=0.02,arrowsize=0.05291667cm 2.0,arrowlength=1.4,arrowinset=0.4,tbarsize=0.07055555cm 5.0]{|*->}(1.84,-3.375036)(1.84,-3.1150358)(4.76,-1.3750359)(5.08,-1.8750359)
\end{pspicture} 
}
\end{center}

\caption{This picture represents the radially symmetric dead core solution of the problem \eqref{rad eq}.}
\end{figure}
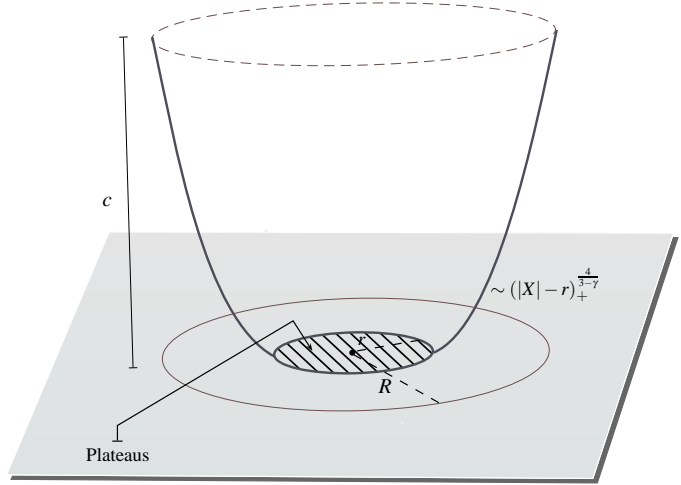

Now, if $v$ is an arbitrary solution to 
$$
	\Delta_\infty v = \lambda v_{+}^{\gamma}, \quad \text{ in } \Omega \subset \mathbb{R}^n,
$$
and $X_0 \in \Omega$ is an interior point, define $\mathfrak{s} \colon (0, \dist(X_0, \partial \Omega)) \to \mathbb{R}_{+}$ by
$$
	\mathfrak{s}(R) := \sup\limits_{B_R(X_0)} v.
$$
If for some $0<R <  \dist(X_0, \partial \Omega)$ , we have
$$
	\mathfrak{s}(R) < \tau(\lambda, \gamma)R^{\frac{4}{3-\gamma}},
$$
then $X_0$ is a plateaus point.  In particular, we can improve Theorem \ref{bernstein} to the following quantitative version:
\begin{theorem}\label{bernstein1}
Let $u$ be a viscosity solution to
\begin{equation}\label{eq thm B1}
	\Delta_\infty u = \lambda u_{+}^\gamma \quad \mbox{in} \quad \mathbb{R}^n.
\end{equation}
If 
\begin{equation}\label{cond thm B1}
	\limsup\limits_{|X| \to \infty} \dfrac{u(X)}{|X|^{\frac{4}{3-\gamma}}} < \sqrt[3-\gamma]{\lambda \cdot \frac{(3-\gamma)^4}{64(1+\gamma)}},
\end{equation}
 then $u\equiv 0$.
\end{theorem}
\begin{proof}
Fixed $R>0$, let us consider $v \colon  \overline{B_R} \to \mathbb{R}$, the solution to the boundary value problem
$$
	\left \{ 
		\begin{array}{rlll}
			\Delta_\infty v &=& \lambda (v^{+})^\gamma &\text{ in } B_R \\
			v &=& \sup\limits_{\partial B_R}u &\text{ on } \partial B_R.
		\end{array}
	\right.
$$
By comparison principle, Lemma \eqref{comparison principle}, $u \leq v$ in $B_R$. It follows by hypothesis \eqref{cond thm B1} that, taking $R\gg 1$ sufficiently large,
\begin{equation}\label{lim rad}
\sup\limits_{\partial B_R}\frac{u(X)}{R^{\frac{4}{3-\gamma}}} \leq \theta \cdot\tau(\lambda,\gamma)
\end{equation}
for some $\theta<1$. For $R\gg 1$, the solution $v = v_R$ is given by
\begin{equation}\label{rad eq3}
v(X)=\tau(\lambda,\gamma) \left(|X|-R + \left[\frac{\sup\limits_{\partial B_R}u}{\tau(\lambda,\gamma)}\right]^{\frac{3-\gamma}{4}} \right)_+^{\frac{4}{3-\gamma}}.
\end{equation} 
Finally, combining \eqref{lim rad} and \eqref{rad eq3}, we get
$$
u(X) \leq \tau(\lambda,\gamma) \left(|X|- (1-\theta^{\frac{3-\gamma}{4}})R \right)_+^{\frac{4}{3-\gamma}},
$$
Letting $R \to \infty$ we conclude the proof of the Theorem.  
\end{proof} 

We conclude by pointing out that Theorem \ref{bernstein1}, as stated,  is sharp in the sense that one cannot remove the strict inequality in \eqref{cond thm B1}.  Indeed, the function
$$
	h(X) := \sqrt[3-\gamma]{\lambda \cdot \frac{(3-\gamma)^4}{64(1+\gamma)}} |X|^{\frac{4}{3-\gamma}},
$$
solves \eqref{eq thm B1} in $\mathbb{R}^n$ and it clearly attains equality in  \eqref{cond thm B1}.

\section{Minimal growth rate and measure estimates} 

In this section we show that the regularity estimate established in Theorem  \ref{thmreg} is indeed sharp. This is done by establishing a competing inequality which controls the minimal growth rate of the solution away from its free boundary.

\begin{theorem}[Nondegeneracy]
\label{nondegthm}  Let $u \in C(\Omega)$ be a nonnegative viscosity solution to 
\begin{equation}\label{lgamma}
\mathcal{L}^{\,\gamma}_{\, \infty} \, u = 0 \quad \mbox{in} \quad \Omega
\end{equation}
and $X_{0} \in \overline{\left\lbrace u> 0 \right\rbrace} \cap \Omega$. There exists a universal constant $c_0>0$, such that 
\begin{eqnarray}\label{nondeg est}
	\sup_{B_{r}\left( X_{0} \right)}u \geq c_0 \cdot r^{\frac{4}{3-\gamma}},
\end{eqnarray}
for all  $0<r <\dist (X_0, \partial \Omega)$.
\end{theorem}

\begin{proof}
By continuity, it suffices to prove \eqref{nondeg est} for points within the set $ \left\lbrace u> 0 \right\rbrace \cap \Omega'$. Initially define
\begin{eqnarray*}
\psi \left( X \right) := c \cdot \vert X - X_{0}\vert^{\alpha},
\end{eqnarray*}
for $\alpha:= \frac{4}{3-\gamma}$ and $c>0$ a constant that will be fixed {\it a posteriori}. By direct computation,
$$
D \psi  (X) = c \alpha |X-X_0|^{\alpha-1}\cdot \frac{X-X_0}{|X-X_0|}.
$$
Continuing, direct computations further yield
\begin{eqnarray*}
D^2 \psi \left( X \right) = c \alpha \left[(\alpha-1)|X-X_0|^{\alpha-2} \cdot \frac{(X-X_0)\otimes (X-X_0)}{|X-X_0|^2}\right. \\
\left. +  |X-X_0|^{\alpha-2}\cdot \left( \text{Id}_{n \times n}-\frac{(X-X_0)\otimes(X-X_0)}{|X-X_0|^2}\right)\right].
\end{eqnarray*}
Therefore, we conclude
 \begin{eqnarray*}
\langle D^2\psi \cdot D\psi,D\psi \rangle \left( X \right) = \left( c \alpha \right)^3(\alpha-1) |X-X_0|^{2(\alpha-1)+(\alpha-2)}
\end{eqnarray*}
and hence, by selecting (and fixing) the constant $c$ within the range
$$
0 < c <  \sqrt[3 - \gamma]{\dfrac{\left( 3 - \gamma \right)^{4} }{64\left( 1+ \gamma  \right)}},
$$
we reach 
$$
\mathcal{L}^{\,\gamma}_{\, \infty} \psi < 0 = \mathcal{L}^{\,\gamma}_{\, \infty}\; u .
$$

Now, for any ball $B_r(X_0) \subset \Omega$, there must exist a point $Z \in \partial B_{r}\left( X_{0}\right)$ such that $\psi \left( Z \right) < u\left( Z \right)$; otherwise, by comparison principle, Lemma \ref{comparison principle}, $\psi \ge u$ in the whole ball $ B_{r}\left( X_{0} \right)$. However, $0 = \psi \left( X_{0} \right) < u\left( X_{0} \right)$. In conclusion, we  can estimate
\begin{eqnarray*}
\sup_{B_{r}\left( X_{0} \right)}u \geq u\left( Y_{r} \right) \geq \psi \left( Y_{r} \right) = c \cdot r^{\frac{4}{3-\gamma}}
\end{eqnarray*}
and the Theorem is proven.
\end{proof}

\begin{corollary}\label{H est}
Given a subdomain $\Omega' \Subset \Omega$, there exists a  constant $\iota >0$ depending on $\|u\|_{L^\infty(\Omega)}, \gamma$ and $\Omega'$ such that for $u \in C(\Omega)$ a nonnegative, bounded viscosity solution to \eqref{lgamma} in $\Omega$, there holds
\begin{eqnarray*}
\frac{\mathscr{L}^{n}\left( B_{r}\left( X_{0}\right) \cap \left\lbrace u> 0 \right\rbrace  \right)}{r^{\,n}} \geq \iota,
\end{eqnarray*}
for any $X_0 \in \partial\{u>0\} \cap \Omega'$ and $0<r \ll 1$. In addition, for a universal constant $0< \sigma_{0} \le 1$, depending only on dimension and $\gamma$, the $(n-\sigma_0)$-Hausdorff measure of $\partial \{u>0\}$ is locally finite.
\end{corollary}
\begin{proof}
In view of Theorem \ref{nondegthm}, for some $r>0$ fixed, it is possible to select a point $Y_0$ such that,
\begin{equation}\label{dens}
u(Y_0)= \sup\limits_{B_r(X_0)} u \geq c_0 \cdot r^{\frac{4}{3-\gamma}}.
\end{equation}
To conclude, we claim that for some $\delta>0$, chosen universally small,
the following inclusion
\begin{equation}\label{inclusion}
B_{\delta \cdot r}(Y_0) \subset  \{u>0\}
\end{equation}
holds. Indeed, by  Theorem \ref{thmreg}, for $Z \in \partial\{u>0\}$, we reach
$$
u(Y_0) \leq C \cdot |Y_0 - Z|^{\frac{4}{3-\gamma}}.
$$
Therefore, by \eqref{dens} and the inequality above, we find
$$
c_0 \cdot r^{\frac{4}{3-\gamma}} \leq C \cdot |Y_0-Z|^{\frac{4}{3-\gamma}}
$$
and so,
$$
\left(\dfrac{c_0}{C}\right)^{\frac{3-\gamma}{4}}\cdot r \leq |\,Y_0-Z|.
$$
Hence, taking $\delta>0$ sufficiently small, the inclusion claimed in \eqref{inclusion} is verified.  

We conclude  with the analysis of the Hausdorff dimension of the free boundary. Let $X_{0} \in \partial \left\lbrace u > 0 \right\rbrace$. From the above reasoning, we can always select 
\begin{eqnarray*}
X'_{0} = \sigma Y_{r} + \left( 1 - \sigma \right)X_{0}.
\end{eqnarray*}
with $0<1-\sigma \ll 1$, such that
\begin{eqnarray*}
B_{\sigma\frac{ r}{2}}\left( X'_{0}  \right) \subset B_{\sigma} \left(  Y_{r} \right) \cap B_{r}\left( X_{0}\right)  \subset B_{r}\left( X_{0}\right) \setminus \partial \left\lbrace u > 0\right\rbrace.
\end{eqnarray*}
Hence the set $\partial \left\lbrace u> 0 \right\rbrace \cap \Omega'$ is $\left( \sigma / 2 \right)$-porous and therefore, by a classical result,  see for instance \cite[ Theorem 2.1]{KR}, the Hausdorff dimension of $\partial \left\lbrace u> 0 \right\rbrace \cap \Omega'$ is at most $n - C\sigma^{n}$ for some dimensional constant $C >0$. 
\end{proof}

\begin{remark}
	The Hausdorff dimension estimate provided by Corollary \ref{H est} assures in particular  that the $\Leb$-Lesbegue measure of the plateaus boundary is zero, but no quantitative information is given on its precise Hausdorff dimension. We believe $\sigma_0 = 1$, and leave this is an open problem.
\end{remark}

\section{The critical equation $\mathcal{L}_\infty^3$}

In this Section we turn our attention to the critical equation obtained as $\gamma \nearrow 3$, that is,
\begin{equation}\label{localeq3}
	\mathcal{L}^{\,3}_{\, \infty} \, u  :=\Delta_{\, \infty}\, u - u^{3} = 0 \quad \mbox{in} \quad \Omega.
\end{equation}
Such an operator is regarded to be critical, as all the estimates established so far deteriores when one let  $\gamma$ converge to $3$. Certainly, one can treat equation \eqref{localeq3} as 
$$
	\Delta_\infty = (u^\delta) \cdot u^{3-\delta},
$$
for any $\delta > 0$. In particular, it follows from Theorem \ref{thmreg} that  if $u$ vanishes at an interior point $\xi \in \Omega$, then $D^n u (\xi) = 0$, for all $n\in \mathbb{N}$. That is,  any zero is an infinite order zero. Under the (very strong) assumption that $u$ is a real analytic function, one could conclude that $u \equiv 0$. 

As mentioned before, Lipschitz regularity is the best local estimate available in the literature for such a solution. Even in the best scenario possible, one could not expect estimates beyond $C^{1,\alpha}$.  Thus assuming $u$ is real analytic would simply be artificial.

Nonetheless, by means of geometric arguments, which explores the scalar invariance of the operator $\mathcal{L}_\infty^3$, we shall  prove that indeed a positive solution to \eqref{localeq3} is prevented to vanish at an interior point. 
 

\begin{theorem} \label{strongmp} Let $u \in C(\Omega)$ be a nonnegative viscosity solution to \eqref{localeq3}. If there exists a point $X \in \Omega$ such that $u(X)=0$, then $u \equiv 0$ in $\Omega$.
\end{theorem}

\begin{proof}
Let us suppose, for the purpose of contradiction,  that the thesis of the theorem fails to hold. With no loss generality we assume $u(0) > 0$ and 
$$
	d:=\dist(0,\{u=0\}) < \dfrac{1}{10}\dist(0,\partial \Omega).
$$ 
By comparison principle $u$ is locally bounded. We now build up the following auxiliary barrier function
\begin{equation}\nonumber
\Phi_{\lambda}(|X|)= 
\left\{
\begin{array}{ccc}
e^{-\lambda(d/2)^2}-\kappa_0 & \mbox{in} & B_{d/2}; \\
e^{-\lambda|X|^2}-\kappa_0 & \mbox{in} & B_d \setminus B_{d/2}; \\
0 & \mbox{in} & \mathbb{R}^n \setminus B_d,
\end{array}
\right. 
\end{equation} 
for $\kappa_0$ such that $\Phi_\lambda(d^+)=0$.  By construction, one easily verifies that
\begin{equation}\label{gradphi}
\inf\limits_{B_d \setminus B_{d/2}}|\nabla \Phi_\lambda| \geq \beta_0
\end{equation}
for some $\beta_0>0$, easily computable if one desires. Moreover, direct computation yields
$$
\mathcal{L}^{\,3}_{\, \infty} \, \Phi_\lambda \geq 0 \quad \mbox{in} \quad B_d \setminus B_{d/2},
$$
provided  $\lambda \gg 1$. The important  observation is that the operator $\mathcal{L}_{\infty}^3$ is invariant under scalar multiplication, that is, for any number $\theta>0$
$$
\mathcal{L}^{\,3}_{\, \infty}(\theta \cdot \Phi_\lambda) = \mathcal{L}^{\,3}_{\, \infty}\, \Phi_\lambda \geq 0 = \mathcal{L}^{\,3}_{\, \infty}\, u  \quad \mbox{in} \quad B_d \setminus B_{d/2}.
$$
In addition, taking $0<\theta \ll 1$ we get
$$
\theta \cdot \Phi_\lambda \leq u \quad \mbox{in} \quad \partial B_d \cup \partial B_{d/2}.
$$
Therefore, by comparison principle, Lemma \ref{comparison principle}, 
\begin{equation}\label{cprin}
\theta \cdot \Phi_\lambda \leq  u \quad \mbox{in} \quad B_d \setminus B_{d/2}.
\end{equation}
On the other hand, equation \eqref{localeq3} can be written as
$$
	\Delta_\infty u = \left [u(X) \right ]\cdot u^{2} = \lambda(X) u^2, 
$$ 
for a bounded Thiele modulus $\lambda(X) = u(X)$. Hence, in view of Remark \ref{rmk1}, we obtain 
$$
	\sup\limits_{B_r(Y_0)} u \leq C \cdot r^{4},
$$
for $Y_0 \in \partial B_d \cap \partial\{u>0\}$. Now, we choose $0<r_0 \ll 1$ such that 
$$
	C \cdot r_0^{4} \le \dfrac{1}{4} \theta \beta_0 \cdot r^{}_0.
$$ 
Finally, by \eqref{gradphi} and \eqref{cprin}, we reach
\begin{eqnarray*}
	0< \theta\beta_0\cdot r_0 &\leq& \sup\limits_{B_{r_0}(Y_0)} \theta \cdot |\Phi_\lambda(X)-\Phi_\lambda(Y_0)| \\
	&\leq&  \sup\limits_{B_{r_0}(Y_0)} \theta\cdot \Phi_\lambda \\
	&\leq& \sup\limits_{B_{r_0}(Y_0)} u \\
	&\le & C \cdot r_0^{4} \\
	&\le &\dfrac{1}{4} \theta \beta_0 \cdot r^{}_0,
\end{eqnarray*}
which gives us a contradiction. The proof of Theorem \ref{strongmp} is complete. 
\end{proof}

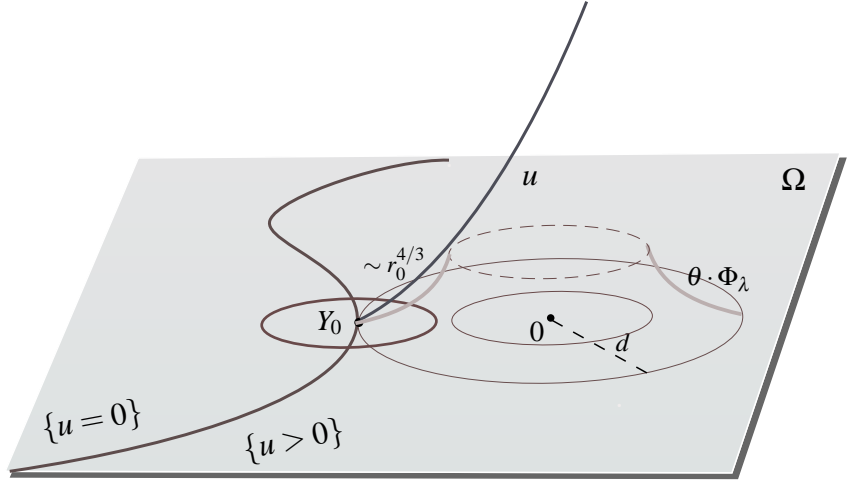
\begin{figure}[h!]
\begin{center}
\scalebox{1} 
{
\begin{pspicture}(0,-3.2182996)(11.421562,3.155876)
\definecolor{color593g}{rgb}{0.8941176470588236,0.8941176470588236,0.8980392156862745}
\definecolor{color593f}{rgb}{0.8666666666666667,0.8862745098039215,0.8862745098039215}
\definecolor{color593e}{rgb}{0.4,0.4,0.4}
\definecolor{color5}{rgb}{0.40784313725490196,0.3058823529411765,0.3058823529411765}
\definecolor{color574}{rgb}{0.9411764705882353,0.9254901960784314,0.9254901960784314}
\definecolor{color579}{rgb}{0.7333333333333333,0.6862745098039216,0.6862745098039216}
\definecolor{color585}{rgb}{0.36470588235294116,0.2980392156862745,0.2980392156862745}
\definecolor{color586}{rgb}{0.40784313725490196,0.28627450980392155,0.28627450980392155}
\definecolor{color589}{rgb}{0.2980392156862745,0.30196078431372547,0.34509803921568627}
\pspolygon[linewidth=0.02,shadowcolor=color593e,linecolor=white,shadow=true,fillstyle=gradient,gradlines=2000,gradbegin=color593g,gradend=color593f,gradmidpoint=1.0](1.77,1.0558761)(0.0,-3.113464)(9.660778,-3.133464)(11.1,1.1265359)
\usefont{T1}{ppl}{m}{n}
\rput(8.164119,-2.2394755){\color{color574}.}
\usefont{T1}{ppl}{m}{n}
\rput(5.9041185,0.9605243){\color{color574}.}
\rput{1.4066979}(-0.02513735,-0.17827968){\psellipse[linewidth=0.01,linecolor=color5,dimen=outer](7.2485213,-1.1129501)(2.5623963,0.8324325)}
\psdots[dotsize=0.1,dotangle=-4.763642](7.251563,-1.0734642)
\usefont{T1}{ptm}{m}{n}
\rput(8.216547,-1.3934642){\small $d$}
\psbezier[linewidth=0.05,linecolor=color579](9.79,-1.024124)(9.05,-0.8641239)(8.67,-0.5841239)(8.535682,-0.10412391)
\usefont{T1}{ptm}{m}{n}
\rput{11.542394}(-0.46670577,-0.82827026){\rput(3.841094,-2.7034643){\large $\{u > 0 \}$}}
\usefont{T1}{ptm}{m}{n}
\rput{10.483721}(-0.42838863,-0.25659156){\rput(1.1610937,-2.4434643){\large $\{u =0\}$}}
\usefont{T1}{ptm}{m}{n}
\rput(4.3214064,-1.1384641){$Y_0$}
\usefont{T1}{ptm}{m}{n}
\rput(10.481093,0.7765358){\large $\Omega$}
\usefont{T1}{ptm}{m}{n}
\rput(7.0741405,-1.2484641){$0$}
\psbezier[linewidth=0.04,linecolor=color585](0.06,-3.1134644)(2.6,-2.7334645)(4.44,-2.073464)(4.66,-1.2534641)(4.88,-0.4334642)(3.4,-0.093464166)(3.52,0.22653583)(3.64,0.54653585)(5.14,1.0465358)(5.9,1.0265357)
\rput{1.4066979}(-0.024212413,-0.17881462){\psellipse[linewidth=0.01,linecolor=color586,dimen=outer](7.2707715,-1.0755461)(1.3386142,0.359646)}
\psdots[dotsize=0.12](4.7,-1.1334641)
\rput{1.4066979}(-0.026648972,-0.11256185){\psellipse[linewidth=0.036,linecolor=color586,dimen=outer](4.5711675,-1.1416575)(1.1793989,0.34184355)}
\psbezier[linewidth=0.04,linecolor=color589](4.72,-1.0934643)(5.87,-0.50412387)(6.93,1.135876)(7.73,3.135876)
\rput{1.4066979}(-0.0026213082,-0.17756745){\psellipse[linewidth=0.012,linecolor=color586,linestyle=dashed,dash=0.17638889cm 0.10583334cm,dimen=outer](7.2307715,-0.19554606)(1.3386142,0.359646)}
\psline[linewidth=0.02cm,linestyle=dashed,dash=0.16cm 0.16cm](7.27,-1.104124)(8.53,-1.8041239)
\psbezier[linewidth=0.05,linecolor=color579](4.6556816,-1.144124)(5.395682,-0.9841239)(5.775682,-0.7041239)(5.91,-0.2241239)
\usefont{T1}{ptm}{m}{n}
\rput(5.1757817,-0.3434641){\footnotesize $\sim r_0^{4/3}$}
\usefont{T1}{ptm}{m}{n}
\rput(6.9782815,0.7865359){\large $u$}
\usefont{T1}{ptm}{m}{n}
\rput(9.493594,-0.55946407){\small $\theta\cdot \Phi_\lambda$}
\end{pspicture}    
}
\end{center}

\caption{Barrier argument in the proof of Theorem \ref{strongmp}.}
\end{figure}

\begin{remark}
We note that in fact the proof of Theorem \ref{strongmp} yields a Hopf-type lemma for the critical equation \eqref{localeq3}.
\end{remark}

\noindent \textsc{Dami\~ao J. Ara\'ujo} \hfill \textsc{Raimundo Leit\~ao}\\
University of Florida/UNILAB\hfill  Universidade Federal do Cear\'a \\
Department of Mathematics \hfill Department of Mathematics \\
Gainvesville, FL-USA 32611-8105 \hfill Fortaleza, CE-Brazil 60455-760\\
\texttt{daraujo@ufl.edu} \hfill
\texttt{rleitao@mat.ufc.br}

\vspace{1cm}
\noindent \textsc{Eduardo V. Teixeira}\\
Universidade Federal do Cear\'a \\
Department of Mathematics  \\
Fortaleza, CE-Brazil 60455-760 \\
\texttt{teixeira@mat.ufc.br}

\end{document}